\documentclass[10pt]{article}
\usepackage{amsmath, amsthm, amssymb, amsfonts, mathrsfs, mathtools}
\usepackage{graphicx}
\usepackage{makeidx}
\usepackage[left=2cm,right=2cm,top=2cm,bottom=2cm]{geometry}
\usepackage{caption}
\usepackage{subcaption}
\usepackage{tikz}
\usetikzlibrary{decorations.pathreplacing}
\tikzstyle{vertex}=[auto=left,circle,draw=black,fill=white, inner sep=1.5]

\usepackage{hyperref}
\hypersetup{colorlinks=true, citecolor={blue}, linkcolor=blue, filecolor=magenta, urlcolor=cyan}

\newtheorem{theorem}{Theorem}[section]
\newtheorem{prop}[theorem]{Proposition}
\newtheorem{lemm}[theorem]{Lemma}

\newtheorem{ex}[theorem]{Example}

\newtheorem{prob}{Problem}[section] 
\title{Chromatic polynomials of signed book graphs}

\author{Deepak Sehrawat and Bikash Bhattacharjya\\
Department of Mathematics\\
Indian Institute of Technology Guwahati, India\\
deepakmath55555@iitg.ac.in ; b.bikash@iitg.ac.in}

\date{}

\begin{document}
\maketitle

\begin{center}{\textbf{Abstract}}\end{center}
For $m \geq 3$ and $n \geq 1$, the $m$-cycle \textit{book graph} $B(m,n)$ consists of $n$ copies of the cycle $C_m$ with one common edge. In this paper, we prove that (a) the number of  switching non-isomorphic signed $B(m,n)$ is $n+1$, and (b) the chromatic number of a signed $B(m,n)$ is either 2 or 3. We also obtain explicit formulas for the chromatic polynomials and the zero-free chromatic polynomials of switching non-isomorphic signed book graphs. 

\noindent{\textbf{Keywords}: signed graph; switching isomorphism; book graph; chromatic number; chromatic polynomial; zero-free chromatic polynomial. 


\section{Introduction}\label{intro}

A \textit{signed graph} $\Sigma=(G, \sigma)$ consists of a graph $G=(V,E)$ and a sign function $\sigma : E \rightarrow \{1,-1\}$. The graph $G$ is called the \textit{underlying} graph of $\Sigma=(G, \sigma)$. The set $\sigma^{-1}(-1)=\{e\in E(G)~|~\sigma(e)=-1 \}$ is called the \textit{signature} of $\Sigma$. That is, by the signature of a signed graph $\Sigma$ we mean the set of negative edges in $\Sigma$. For convenience, we call $\sigma$ itself a signature and write $|\sigma|=|\sigma^{-1}(-1)|$. 

A cycle $C$ in $\Sigma$ is called \textit{positive} if the product $\prod_{e\in E(C)} \sigma(e)=1$ and \textit{negative}, otherwise. A signed graph $\Sigma$ is said to be \textit{balanced} if every cycle of the graph is positive and \textit{unbalanced}, otherwise. The notion of signed graphs and balance was introduced by Harary in~\cite{Harary}.

\textit{Switching} $\Sigma$ by a vertex $v$ changes the signs of all the edges incident to $v$. We say that a signed graph $\Sigma_1$ is \textit{switching equivalent}, or simply \textit{equivalent}, to another signed graph $\Sigma_2$ if one can be obtained from the other by a sequence of switchings. If we switch $\Sigma$ by every vertex of a subset $X \subseteq V(\Sigma)$ then the resulting signed graph is denoted by $\Sigma^{X}$. It is easy to see that switching defines an equivalence relation on the set of all signed graphs with an underlying graph $G$. For a signed graph $\Sigma$, its switching equivalence class is denoted by $[\Sigma]$. As switching operation preserve the product of signs on each cycle, any property of signed graphs that depends only on the signs of the cycles is invariant for all signed graphs belonging to $[\Sigma]$.

One of the fundamental theorems in the theory of signed graphs is that the set of negative cycles uniquely determines the equivalence class to which a signed graph belongs. More precisely, we state the following theorem.

\begin{theorem}[\cite{Zaslavsky}]
\label{Signature}
Two signed graphs $\Sigma_1$ and $\Sigma_2$ are switching equivalent if and only if they have the same set of negative cycles.
\end{theorem}

Two signed graphs $\Sigma_1 = (G, \sigma_{1})$ and $\Sigma_2=(H, \sigma_{2})$ are said to be \textit{isomorphic}, denoted $\Sigma_{1} \cong \Sigma_{2}$, if there exists a graph isomorphism $\psi : V(G) \rightarrow V(H)$ which preserve the edge signs. The signed graphs $\Sigma_{1}$ and $\Sigma_{2}$ are said to be \textit{switching isomorphic}, denoted $\Sigma_{1} \sim \Sigma_{2}$, if $\Sigma_{1}$ is isomorphic to $\Sigma_{2}^{X}$ for some $X \subseteq V(\Sigma_2)$. Two signed graphs having the same underlying graph are called \textit{automorphic} to each other if they are isomorphic.

In~\cite{Zaslavsky2}, Zaslavsky initiated  the study of vertex colorings of signed graphs. A \textit{coloring} of a signed graph $\Sigma= (G, \sigma)$ in $2k+1$ \textit{signed colors} is a mapping $c: V(G) \rightarrow \{-k,\ldots,-1,0,1,\ldots,k\}$. Similarly, a \textit{zero-free coloring} of a signed graph $\Sigma= (G, \sigma)$ in $2k$ \textit{signed colors} is a mapping $c: V(G) \rightarrow \{-k,\ldots,-1,1,\ldots,k\}$. A coloring $c$ of $\Sigma$ is said to be \textit{proper} if $c(x) \neq \sigma(e) c(y)$ for every edge $e=xy$, where $\sigma(e)$ is the sign of the edge $e$. In other words, a coloring of a signed graph is proper if the colors of the vertices incident to a positive edge are not equal, while those incident to a negative edge are not equal in absolute values whenever they are of opposite signs. 

The \textit{chromatic polynomial} $\chi_{\Sigma}(\lambda)$ of a signed graph $\Sigma$ is the function whose value, for odd positive arguments $\lambda=2k+1$, equals the number of proper colorings of $\Sigma$ in $2k+1$ signed colors. The \textit{zero-free chromatic polynomial} $\chi^{b}_{\Sigma}(\lambda)$ of a signed graph $\Sigma$ is the function, where  $\chi^{b}_{\Sigma}(2k)$ counts the zero-free proper colorings in $2k$ signed colors. The chromatic polynomial of a graph $G$ is denoted by $\chi_{G}(\lambda)$. The \textit{chromatic number} $\chi(\Sigma)$ of $\Sigma$ is defined to be the minimum number of the set
\[\left\{2k+1~:~ \chi_{\Sigma}(2k+1)>0 \right\} \cup \left\{2r~:~ \chi_{\Sigma}^b(2r)>0 \right\}.\]
There is a stronger conclusion if $\Sigma$ is balanced. More precisely, we have the following lemma.

\begin{lemm}[\cite{Zaslavsky2}] \label{identical chro}
If $\Sigma = (G,\sigma)$ is balanced, then $\chi_{G}(\lambda) = \chi_{\Sigma}(\lambda) = \chi^{b}_{\Sigma}(\lambda)$.
\end{lemm}

Let the signed graphs $\Sigma$ and $\Sigma'$ be switching equivalent. It is clear that a coloring $c$ of $\Sigma$ is proper if and only if the coloring $c'$ of $\Sigma'$ is proper, where $c'$ is obtained from $c$ after negating the colors of the vertices by which $\Sigma$ is switched. Thus the chromatic number of a signed graph remains invariant under switching operation.

In~\cite{Zaslavsky2}, Zaslavsky proved that the chromatic number and the chromatic polynomials of a signed graph are invariant under switching operation. In perfect analogy to ordinary graph coloring theory, we have the following theorem.

\begin{theorem}[\cite{Zaslavsky2}]
If $\Sigma$ is a signed graph on $n$ vertices, then $\chi_{\Sigma}(\lambda)$ and $\chi_{\Sigma}^b(\lambda)$ are monic polynomial functions of $\lambda$ of degree $n$.
\end{theorem}

Let $e$ be a positive edge in $\Sigma=(G,\sigma)$. The \textit{edge-contraction} $\Sigma /e$ is obtained by identifying the end vertices of $e$ and deleting $e$. We also have the signed analogue of edge deletion-contraction formula for the chromatic polynomial of simple graph.  

\begin{theorem}[\cite{Zaslavsky2}] \label{E-C}
Let $\Sigma$ be a signed graph and $e$ be a positive edge in $\Sigma$. Then  
$$\chi_{\Sigma}(\lambda) = \chi_{\Sigma \setminus e}(\lambda) - \chi_{\Sigma /e}(\lambda)~\text{and}~\chi_{\Sigma}^b(\lambda) = \chi_{\Sigma \setminus e}^b(\lambda) - \chi_{\Sigma /e}^b(\lambda).$$
\end{theorem} 

Chromatic polynomials of signed graphs has been less studied. However, in order to compute both the chromatic polynomials (\textit{i.e.,} the chromatic polynomial as well as the zero-free chromatic polynomial) of a signed graph, a remarkable work was done by Mathias Beck and his team. In~\cite{Beck}, they published a SAGE code which produces both the chromatic polynomials as output when a signed graph is given as input. Using that code they have presented explicit formulas for both kind of chromatic polynomials of six switching non-isomorphic signed Petersen graphs.  

The aim of this paper is to obtain formulas for the chromatic polynomials as well as the zero-free chromatic polynomials of switching non-isomorphic signed book graphs. For $m \geq 3$ and $n \geq 1$, the $m$-cycle \textit{book graph} $B(m,n)$ consists of $n$ copies of the cycle $C_m$ with one common edge. The copies of the cycle $C_m$ are called the \textit{pages} of $B(m,n)$. Let  $V(B(m,n)) = \{u,v\}\cup \{u_{j}^i~|~1 \leq i \leq n,~~ 1 \leq j \leq m-2 \}$, and let $uv$ be the common edge to the cycles $C_{m}^{i}$, where $C_{m}^{i} = uu_{1}^{i}u_{2}^{i}u_{3}^{i}...u_{m-3}^{i}u_{m-2}^{i}vu$, for $1 \leq i \leq n$. For example, the cycle $C_{4}^{1}$ in $B(4,3)$ is the cycle $uu_{1}^{1}u_{2}^{1}vu$, where the graph $B(4,3)$ is shown in Figure~\ref{BG1}.

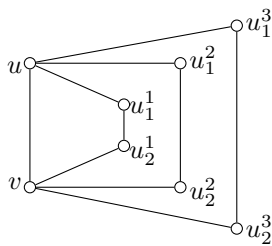
\begin{figure}[ht]
\centering
\begin{tikzpicture}[scale=0.5]
\node[vertex] (v1) at (9,2) {};
\node [below] at (8.6,2.5) {$v$};
\node[vertex] (v2) at (9,5.3) {};
\node [below] at (8.6,5.6) {$u$};
\node[vertex] (v3) at (11.5,3.1) {};
\node [below] at (12.0,3.6) {$u_{2}^{1}$};
\node[vertex] (v4) at (11.5,4.2) {};
\node [below] at (12.0,4.8) {$u_{1}^{1}$};
\node[vertex] (v5) at (13,2) {};
\node [below] at (13.6,2.5) {$u_{2}^{2}$};
\node[vertex] (v6) at (13,5.3) {};
\node [below] at (13.6,6) {$u_{1}^{2}$};
\node[vertex] (v7) at (14.5,0.9) {};
\node [below] at (15.1,1.5) {$u_{2}^{3}$};
\node[vertex] (v8) at (14.5,6.3) {};
\node [below] at (15.1,7) {$u_{1}^{3}$};
\foreach \from/\to in {v1/v2,v1/v3,v2/v4,v1/v5,v2/v6,v1/v7,v2/v8,v3/v4,v5/v6,v7/v8} \draw (\from) -- (\to);
\end{tikzpicture}
\caption{The book graph $B(4,3)$}
\label{BG1}
\end{figure}

It is clear that any permutation of the $n$ copies of the cycle $C_m$ or any permutation of $u$ and $v$ determines an automorphism of $B(m,n)$, and vice-versa. Thus an automorphism of $B(m,n)$ can only permute the vertices $u$ and $v$, and permute its $n$ pages.

In Section~\ref{book} of the paper, it is shown for given $m \geq 3$ and $n \geq 1$ that the number of switching non-isomorphic signed $B(m,n)$ is $n+1$. It is also shown that the chromatic number of a signed $B(m,n)$ is either 2 or 3. Finally, recursive formulas for the chromatic polynomials and the zero-free chromatic polynomials of signed book graphs are presented in Section~\ref{chrom-poly} and Section~\ref{balanced-poly}, respectively.   

\section{The chromatic number of signed book graphs}\label{book}

Note that every signed cycle is switching equivalent to a signed cycle whose signature is either empty or of size one. We will use this fact in the proof of the following theorem.

\begin{theorem}\label{prop1}
\label{theorem1}
Let $m\geq 3, n\geq1$ be fixed and $(B(m,n),\sigma)$ a signed book graph. Then $(B(m,n),\sigma)$ is equivalent to $(B(m,n),\tau)$, where $\tau\subseteq \{uu_{1}^{1},\ldots,uu_{1}^{n}\}$. Further, the number of switching non-isomorphic signed $B(m,n)$ is $n+1$.  
\end{theorem}
\begin{proof}
Let $(B(m,n),\sigma)$ be a signed book graph. By suitable switchings, if needed, we can make each negative edge of $B(m,n)$ incident to $u$. If the edge $uv$ is negative, switching $u$ will make it positive. Thus we get a signature $\tau$ equivalent to $\sigma$ such that $\tau\subseteq \{uu_{1}^{1},\ldots,uu_{1}^{n}\}$.

Note that if $\sigma,\tau\subseteq \{uu_{1}^{1},\ldots,uu_{1}^{n}\}$ and $|\sigma|=|\tau|$, then an one-one correspondence between $\sigma$ and $\tau$ determines an isomorphism between $(B(m,n),\sigma)$ and $(B(m,n),\tau)$. However, if $|\sigma|\neq |\tau|$, then $(B(m,n),\sigma)$ cannot be switching isomorphic to $(B(m,n),\tau)$ because the number of negative cycles $C_m$ are different.  Therefore, the number of switching non-isomorphic signed $B(m,n)$ is $n+1$
\end{proof}

Let $\sigma_0=\emptyset$. For each $1 \leq l \leq n$, let $\sigma_{l} = \{uu_{1}^{1}, uu_{1}^{2}, \cdots , uu_{1}^{l}\}$. By the preceding theorem, $\sigma_0,\sigma_1,\ldots, \sigma_n$ are switching non-isomorphic signatures of $B(m,n)$.

It is proved in~\cite[Proposition 2.4]{Macajova} that $\chi_{(C_m, \sigma)} \leq 3$, where $(C_m, \sigma)$ is a signed cycle. Therefore we have $\chi_{(B(m,1), \sigma)} \leq 3$. In the following theorem, we assume that $n \geq 2$.  

\begin{theorem}\label{chro no of book graph}
Let $m\geq 3,n\geq 2$ and $0\leq l\leq n$. Then
\[\chi(B(m,n), \sigma_{l})= \left\{ \begin{array}{ll}
2 & \textrm{if $m$ is odd and $l=n$,}\\
2 & \textrm{if $m$ is even and $l=0$,}\\
3 & \textrm{otherwise.} \end{array}\right. \]
\end{theorem}
\begin{proof}  In $(B(2k+1,n), \sigma_{n})$, assign colors $1$ and $-1$ to $u$ and $v$, respectively. For each $1 \leq r \leq n$,  assign colors $1,-1,1,...,-1,1$ to the vertices $u_{1}^{r},u_{2}^{r},u_{3}^{r},...,u_{2k-2}^{r},u_{2k-1}^{r}$, respectively. This gives a proper 2-coloring of $(B(2k+1,n), \sigma_{n})$. 

In $(B(2k,n), \sigma_{0})$, assign the colors $1$ and $-1$ to the vertices $u$ and $v$, respectively. For each $1 \leq r \leq n$, the vertices $u_{1}^{r},u_{2}^{r},u_{3}^{r},u_{4}^{r},...,u_{2k-3}^{r},u_{2k-2}^{r}$ are colored with $-1,1,-1,1,...,-1,1$, respectively. This gives a proper 2-coloring of $(B(2k,n), \sigma_{0})$. 

Note that the odd cycle $C_{2k+1}^{n} = uu_{1}^{n}u_{2}^{n}u_{3}^{n}...u_{2k-2}^{n}u_{2k-1}^{n}vu$ is positive in $(B(2k+1,n), \sigma_{l})$, where $0 \leq l \leq n-1$. It is well known that at least three colors are needed to properly color the vertices of a positive odd cycle. Therefore, $\chi(B(2k+1,n), \sigma_l) \geq 3$. We now give a proper 3-coloring of $(B(2k+1,n), \sigma_l)$. Assign colors 0 and $-1$ to $u$ and $v$, respectively. For each $1 \leq r \leq n$, the vertices $u_{1}^{r},u_{2}^{r},u_{3}^{r},...,u_{2k-2}^{r},u_{2k-1}^{r}$ are colored with $1,-1,1,...,-1,1$, respectively. This assignment is a proper 3-coloring of $(B(2k+1,n), \sigma_l)$, where $0 \leq l \leq n-1$. 

Finally, $(B(2k,n), \sigma_{l})$ contains a negative $2k$-cycle for each $1 \leq l \leq n$. The fact that an unbalanced even cycle needs at least 3 colors to have a proper coloring implies that $\chi(B(2k,n), \sigma_{l}) \geq 3$. To complete the proof it suffices to give a proper 3-coloring of $(B(2k,n), \sigma_{l})$. Assign colors 0 and $1$ to $u$ and $v$, respectively. For each $1 \leq r \leq n$, the vertices $u_{1}^{r},u_{2}^{r},u_{3}^{r},...,u_{2k-2}^{r},u_{2k-1}^{r}$ are colored with $1,-1,1,...,1,-1$, respectively. This assignment is a proper 3-coloring of $(B(2k,n), \sigma_{l})$. 
\end{proof}

As the chromatic number of a signed graph is invariant under switching operation, by Theorem~\ref{chro no of book graph}, we conclude that the chromatic number of a signed book graph is either 2 or 3.


\section{Chromatic polynomials of signed book graphs}\label{chrom-poly}

For our convenience, we write $B_{l}(m,n)$ to denote the signed book graph $(B(m,n), \sigma_l)$. Also, we write $B^{uv}(m,n)$ to denote $(B(m,n), \{uv\})$. It is clear that $B_{n}(m,n) \sim B^{uv}(m,n)$. Since the chromatic polynomials of a signed graph are switching invariant, it is sufficient to the determine the chromatic polynomials of $B^{uv}(m,n)$ and $B_l(m,n)$ for $0 \leq l \leq n-1$. 

Let an unbalanced cycle on $n$ vertices be denoted by $C_n^-$. For instance, $C_2^{-}$ is shown in Figure~\ref{C_2}. 

\begin{figure}[ht]
\centering
\begin{tikzpicture}
\node[vertex] (v1) at (0,0) {};
\node [below] at (-0.1,-0.01) {u};
\node[vertex] (v2) at (2,0) {};
\node [below] at (2.1,-0.01) {v};
\draw [dashed] (2,0) to[out=-120,in=-60] node [sloped,above] { } (0,0);
\draw (2,0) to[out=120,in=60] node [sloped,above] { } (0,0);
\end{tikzpicture} 
\caption{An unbalanced cycle of length two} \label{C_2}
\end{figure}
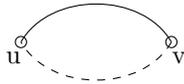

\begin{ex}\label{ex2}
\rm{Let the colors $-k,\ldots,-1,0,1\ldots,k$ be available. The number of proper colorings of $C_2^{-}$ with these $2k+1$ colors, in which one of $u$ or $v$ is colored with 0, is  $4k$. Else, the number of proper colorings of $C_2^{-}$ is $2k(2k-2)$. Thus $\chi_{C_{2}^{-}}(2k+1)  =4k+2k(2k-2) =(2k)^2$. }
\end{ex}
From Example~\ref{ex2}, we find that $\chi_{C_{2}^{-}}(\lambda) =(\lambda - 1)^2$. In~\cite{Beck}, it is proved that $\chi_{C_{3}^{-}}(\lambda) = (\lambda - 1)^3$. Now we give the formula of $\chi_{C_{n}^{-}}(\lambda)$ for all $n$.

\begin{lemm} \label{C-}
Let $C_{n}^{-}$ be an unbalanced cycle, where $n \geq 2$. Then $\chi_{C_{n}^{-}}(\lambda) = (\lambda-1)^{n}$. 
\end{lemm}
\begin{proof}
We prove this lemma by induction on $n$. If $n = 2$, the result is true by Example~\ref{ex2}. Let us assume that the result holds for all $n \leq r-1$, where $r \geq 3$. We shall prove that the result is also true for $n = r$. If $e$ is a positive edge of $C_{r}^{-}$ then by edge deletion-contraction formula, we get 
\begin{equation} \label{eq1}
\chi_{C_{r}^{-}}(\lambda) = \chi_{P_{r}}(\lambda) - \chi_{C_{r-1}^{-}}(\lambda).
\end{equation}
It is well known that $\chi_{P_r}(\lambda) = \lambda(\lambda - 1)^{r-1}$, and by induction hypothesis we have $\chi_{C_{r-1}^{-}}(\lambda) = (\lambda-1)^{r-1}$. Therefore Equation~(\ref{eq1}) gives
\begin{equation*}
 \chi_{C_{r}^{-}}(\lambda) = \lambda(\lambda - 1)^{r-1} - (\lambda - 1)^{r-1} = (\lambda - 1)^{r}.
\end{equation*}
Thus the proof follows by induction.
\end{proof}

It is well known that $\chi_{C_m}(\lambda) = (\lambda-1)^{m} + (-1)^{m}(\lambda-1)$. Using the function $\chi_{C_m}(\lambda)$, a polynomial function $\gamma_m$ of degree $m-2$ is defined as 
\begin{equation} \label{eq of gamma_m}
\gamma_m = \frac{\chi_{C_m}(\lambda)}{\lambda(\lambda-1)} = \frac{(\lambda-1)^{m-1} - (-1)^{m-1}}{\lambda}.
\end{equation}

The function $\gamma_m$ can also be expressed as the finite geometric series. More precisely, we have 

\begin{equation}
 \gamma_m = \frac{(\lambda - 1)^{m-1}+(-1)^m}{\lambda} =  \sum_{i=0}^{m-2} (-1)^{i} (\lambda - 1)^{(m-2)-i}
\end{equation}

We will see that the function $\gamma_m$ is very much useful for obtaining the explicit formulas of both kind of chromatic polynomials of signed book graphs. 

Let $\Sigma$ be a given signed graph. Construct the signed graph $\Sigma_{t+1}$ by attaching the path \linebreak[4] $P_{t+1}:=uu_1u_2\ldots u_t$ to a vertex $u$ of $\Sigma$. If the chromatic polynomial of $\Sigma$ is known then we can compute the chromatic polynomial of $\Sigma_{t+1}$ using the following lemma.

\begin{lemm}\label{Path_add}
Let $\Sigma$ be a signed graph. Then $\chi_{\Sigma_{t+1}}(\lambda) = (\lambda -1)^{t} \chi_{\Sigma}(\lambda)$.
\end{lemm}
\begin{proof}
We prove this lemma by induction on $t$. Note that $\chi_{P_{t+1}}(\lambda) =\lambda(\lambda - 1)^{t}$. Let $t = 1$, and $e_1 = uu_{1}$ be attached to a vertex $u$ of $\Sigma$. Using edge deletion-contraction formula on $e_1$, we get 
\begin{equation*}
\chi_{\Sigma_{2}}(\lambda) = \lambda \chi_{\Sigma}(\lambda) - \chi_{\Sigma}(\lambda) = (\lambda - 1)  \chi_{\Sigma}(\lambda).
\end{equation*}
Now assume that the result holds for $t = r-1$, that is, $\chi_{\Sigma_{r}}(\lambda) = (\lambda - 1)^{r-1}  \chi_{\Sigma}(\lambda)$, where $r \geq 3$. Now let $t = r$ and let $e = uu_{1}$. Using edge deletion-contraction formula on the edge $e$ of $\Sigma_{r+1}$, we get $\chi_{\Sigma_{r+1}}(\lambda) = \chi_{P_{r} \cup \Sigma}(\lambda) - \chi_{\Sigma_{r}}(\lambda)$, where $P_r$ and $\Sigma$ are disjoint. Thus we have
\begin{align*}
\chi_{\Sigma_{r+1}}(\lambda) & = \lambda(\lambda - 1)^{r-1}  \chi_{\Sigma}(\lambda) - \chi_{\Sigma_{r}}(\lambda) \\
 & = \lambda(\lambda - 1)^{r-1}  \chi_{\Sigma}(\lambda) - (\lambda - 1)^{r-1}  \chi_{\Sigma}(\lambda) \\
 & = (\lambda - 1)^{r}  \chi_{\Sigma}(\lambda).
\end{align*}
Hence the proof follows by induction. 
\end{proof}

Replace the edge $uv$ of $B(m,n)$ by an unbalanced cycle of length two, and denote the graph so obtained by $B_{m}^{n}$. For example, $B_{4}^{3}$ and $B_{m}^{1}$ are shown in Figure~\ref{Figure5}. As a convention, write $B_{2}^{1} = C_{2}^{-}$ so that $\chi_{B_{2}^{1}}(\lambda) = (\lambda-1)^{2}$.

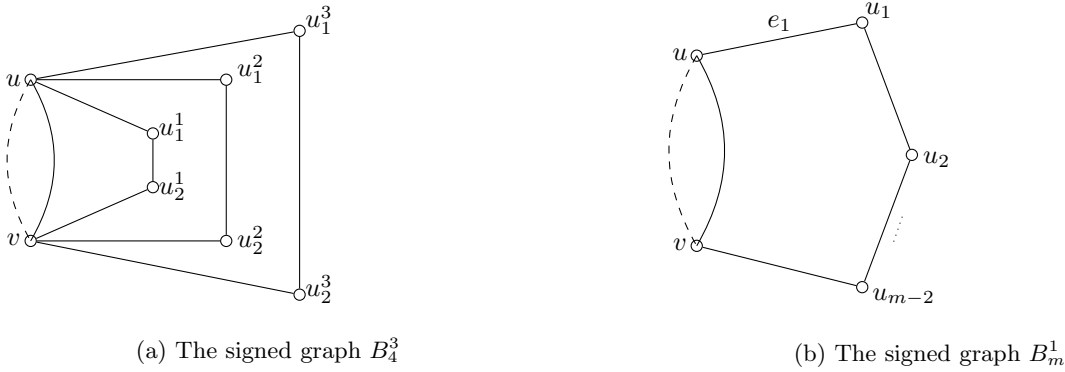
\begin{figure}[ht]
\centering
\begin{subfigure}{0.45\textwidth}
\begin{tikzpicture}[scale=0.65]
\node[vertex] (v1) at (9,2) {};
\node [below] at (8.65,2.35) {$v$};
\node[vertex] (v2) at (9,5.3) {};
\node [below] at (8.65,5.6) {$u$};
\node[vertex] (v3) at (11.5,3.1) {};
\node [below] at (11.9,3.6) {$u_{2}^{1}$};
\node[vertex] (v4) at (11.5,4.2) {};
\node [below] at (11.9,4.8) {$u_{1}^{1}$};
\node[vertex] (v5) at (13,2) {};
\node [below] at (13.5,2.5) {$u_{2}^{2}$};
\node[vertex] (v6) at (13,5.3) {};
\node [below] at (13.5,6) {$u_{1}^{2}$};
\node[vertex] (v7) at (14.5,0.9) {};
\node [below] at (14.9,1.5) {$u_{2}^{3}$};
\node[vertex] (v8) at (14.5,6.3) {};
\node [below] at (14.9,7) {$u_{1}^{3}$};

\foreach \from/\to in {v1/v3,v2/v4,v1/v5,v2/v6,v1/v7,v2/v8,v3/v4,v5/v6,v7/v8} \draw (\from) -- (\to);

\draw [dashed] (9,5.3) to[out=-120,in=120] node [sloped,above] { } (9,2);
\draw (9,5.3) to[out=-60,in=60] node [sloped,above] { } (9,2);

\end{tikzpicture}
\caption{The signed graph $B_{4}^{3}$}\label{unbal_2}
\end{subfigure}
\hfill
\begin{subfigure}{0.45\textwidth}
\begin{tikzpicture}[scale=1.1]
\node[vertex] (v1) at (9,2) {};
\node [below] at (8.8,2.2) {$v$};
\node[vertex] (v2) at (9,4.3) {};
\node [below] at (8.8,4.5) {$u$};
\node[vertex] (v7) at (11,1.5) {};
\node [below] at (11.5,1.6) {$u_{m-2}$};
\node[vertex] (v8) at (11,4.7) {};
\node [below] at (11.2,5.05) {$u_{1}$};
\node[vertex] (v9) at (11.6,3.1) {};
\node [below] at (11.9,3.25) {$u_{2}$};
\node [below] at (10.0,4.9) {$e_1$};
\draw[dotted](11.38,2.05) -- (11.5,2.40);

\foreach \from/\to in {v1/v7,v2/v8,v7/v9,v8/v9} \draw (\from) -- (\to);
\draw [dashed] (9,4.3) to [out=-120,in=120] node [sloped,above] { } (9,2);
\draw (9,4.3) to[out=-60,in=60] node [sloped,above] { } (9,2);

\end{tikzpicture}
\caption{The signed graph $B_{m}^{1}$} \label{figure6.1}
\end{subfigure}
\caption{The signed graphs $B_{4}^{3}$ and $B_{m}^{1}$}\label{Figure5}
\end{figure}

\begin{lemm}\label{B_m}
For $m \geq 2$, the chromatic polynomial of $B_{m}^{1}$ is given by 
$$\chi_{B_{m}^{1}}(\lambda) = (\lambda - 1)^{2} \gamma_{m}.$$
\end{lemm}
\begin{proof}
We prove this lemma by induction on $m$. Consider $B_{m}^{1}$ as given in Figure~\ref{Figure5}(b). For $m=2$, the result is true by Example~\ref{ex2} since $\gamma_2 = 1$. Assume that the result is true for $m = r-1$, where $r \geq 3$. That is, 
\begin{equation} \label{eq2}
\chi_{B_{r-1}^{1}}(\lambda) =  (\lambda - 1)^{2}\gamma_{r-1}.
\end{equation}
An application of edge deletion-contraction on the edge $e_1 = uu_{1}$ of $B_{r}^{1}$ is shown in Figure~\ref{Un_1}. Using Lemma~\ref{Path_add}, the chromatic polynomial of the graph in the middle of Figure~\ref{Un_1} can be computed easily, since $\chi_{C_{2}^{-}}(\lambda)$ is known. The chromatic polynomial of the third graph of Figure~\ref{Un_1} is given in Equation~(\ref{eq2}). Therefore 
\begin{align*}
\chi_{B_{r}^{1}}(\lambda) & = (\lambda - 1)^{r-2}  \chi_{C_{2}^{-}}(\lambda) -  (\lambda - 1)^{2} \gamma_{r-1} \\
 & = (\lambda - 1)^{r} - (\lambda - 1)^{2} \frac{(\lambda - 1)^{r-2} - (-1)^{r-2}}{\lambda} \\
 & = (\lambda - 1)^{2}  \frac{ (\lambda - 1)^{r-1} - (-1)^{r-1}}{\lambda}\\
 & = (\lambda - 1)^{2}\gamma_{r}.
\end{align*}
This completes the proof.
\end{proof}

\begin{figure}[h]
\centering
\begin{subfigure}{0.24\textwidth}
\begin{tikzpicture}[scale=0.55]
\node[vertex] (v1) at (9,2) {};
\node [below] at (8.6,2.3) {$v$};
\node[vertex] (v2) at (9,4.3) {};
\node [below] at (8.6,4.6) {$u$};
\node[vertex] (v7) at (11,1.5) {};
\node [below] at (11.75,1.75) {$u_{r-2}$};
\node[vertex] (v8) at (11,4.7) {};
\node [below] at (11.5,5.2) {$u_{1}$};
\node[vertex] (v9) at (11.6,3.1) {};
\node [below] at (12.1,3.35) {$u_{2}$};
\node [below] at (14.0,3.5) {$\longrightarrow$};
\node [below] at (10,5.21) {$e_1$};
\draw[dotted](11.5,2.0) -- (11.62,2.35);

\foreach \from/\to in {v1/v7,v2/v8,v7/v9,v8/v9} \draw (\from) -- (\to);

\draw [dashed] (9,4.3) to node [sloped,above] { } (9,2);
\draw (9,4.3) to[out=-60,in=60] node [sloped,above] { } (9,2);

\end{tikzpicture}
\end{subfigure}
\begin{subfigure}{0.24\textwidth}
\begin{tikzpicture}[scale=0.55]
\node[vertex] (v1) at (9,2) {};
\node [below] at (8.6,2.3) {$v$};
\node[vertex] (v2) at (9,4.3) {};
\node [below] at (8.6,4.6) {$u$};
\node[vertex] (v7) at (11,1.5) {};
\node [below] at (11.75,1.75) {$u_{r-2}$};
\node[vertex] (v8) at (11,4.7) {};
\node [below] at (11.5,5.2) {$u_{1}$};
\node[vertex] (v9) at (11.6,3.1) {};
\node [below] at (12.1,3.35) {$u_{2}$};
\node [below] at (14.0,3.5) {$-$};
\draw[dotted](11.5,2.0) -- (11.62,2.35);

\foreach \from/\to in {v1/v7,v7/v9,v8/v9} \draw (\from) -- (\to);

\draw [dashed] (9,4.3) to node [sloped,above] { } (9,2);
\draw (9,4.3) to[out=-60,in=60] node [sloped,above] { } (9,2);
\end{tikzpicture}
\end{subfigure}
\begin{subfigure}{0.24\textwidth}
\begin{tikzpicture}[scale=0.55]
\node[vertex] (v1) at (9,2) {};
\node [below] at (8.6,2.2) {$v$};
\node[vertex] (v2) at (9,4.3) {};
\node [below] at (8.6,4.5) {$u$};
\node[vertex] (v7) at (11,1.5) {};
\node [below] at (11.8,1.65) {$u_{r-2}$};
\node[vertex] (v8) at (11,4.7) {};
\node [below] at (11.5,5.1) {$u_{2}$};
\node[vertex] (v9) at (11.6,3.1) {};
\node [below] at (12.05,3.45) {$u_{3}$};
\draw[dotted](11.4,2.0) -- (11.62,2.55);

\foreach \from/\to in {v1/v7,v2/v8,v7/v9,v8/v9} \draw (\from) -- (\to);

\draw [dashed] (9,4.3) to node [sloped,above] { } (9,2);
\draw (9,4.3) to[out=-60,in=60] node [sloped,above] { } (9,2);

\end{tikzpicture}
\end{subfigure}
\caption{An application of edge deletion-contraction on $B_{r}^{1}$} \label{Un_1}
\end{figure}
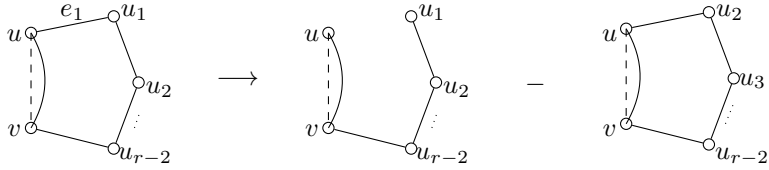

Now we give the formula for the chromatic polynomial of the signed graph $B_{m}^{n}$, where $n \geq 2$. To do so we use Theorem~\ref{E-C}, Lemma~\ref{Path_add}, and Lemma~\ref{B_m}. 

\begin{prop}\label{prop_1}
For $n \geq 2$, the chromatic polynomial of the signed graph $B_{m}^{n}$ is given by 
\begin{equation*}
 \chi_{B_{m}^{n}}(\lambda) = (\lambda - 1)^{2} \gamma_{m}^{n}.
\end{equation*}  
\end{prop} 
\begin{proof}
We sequentially use edge deletion-contraction formula on the edges $uu_{1}^{n},uu_{2}^{n},uu_{3}^{n}$, and so on. This process gives 

\begin{align*}
\chi_{B_{m}^{n}}(\lambda) & = \left[ \sum_{i=0}^{m-2} (-1)^{i} (\lambda - 1)^{(m-2)-i} \right] \chi_{B_{m}^{n-1}}(\lambda) \\
 & = \gamma_m  \chi_{B_{m}^{n-1}}(\lambda).
\end{align*} 
 
Since $\chi_{B_{m}^{1}}(\lambda)$ is known from Lemma~\ref{B_m}, after $n-2$ recursions, we find that 
$$\chi_{B_{m}^{n}}(\lambda) = (\lambda - 1)^{2} \gamma_{m}^{n}.$$ 
This completes the proof.
\end{proof}

It is well known that if $G$ and $H$ are two simple graphs such that $G \cap H$ is a complete graph, then the chromatic polynomial of $G \cup H$ is given by  
\begin{equation} \label{eq3}
\chi_{G \cup H}(\lambda) = \frac{\chi_{G}(\lambda) \cdot \chi_{H}(\lambda)}{\chi_{G \cap H}(\lambda)}.
\end{equation}

Using Equation~(\ref{eq3}), we now determine the chromatic polynomial of an unsigned book graph.

\begin{theorem} \label{B_0}
The chromatic polynomial of $B(m,n)$ is given by 
\begin{equation*}
\chi_{B(m,n)}(\lambda) = \lambda(\lambda-1) \gamma_{m}^{n}.
\end{equation*}
\end{theorem}
\begin{proof}
We prove this theorem by induction on $n$. Clearly $B(m,1) = C_m$, so by the definition of $\gamma_m$ we have
\begin{equation} \label{eq4}
\chi_{B(m,1)}(\lambda) = \chi_{C_m}(\lambda) = \lambda(\lambda-1) \gamma_{m}.   
\end{equation}

Further, the graph $B(m,2)$ is the union of two $m$-cycles whose intersection is $K_2$. Thus by Equation~(\ref{eq3}), we get 
\begin{equation*}
\chi_{B(m,2)}(\lambda) = \lambda(\lambda-1) \gamma_{m}^{2}.
\end{equation*} 

This shows that the result is true for $n=1$ and $n = 2$. Let us assume that the result is true for  $n = r-1$, where $r \geq 3$. That is, 
\begin{equation} \label{eq5}
\chi_{B(m,r-1)}(\lambda) = \lambda(\lambda-1) \gamma_{m}^{r-1}.
\end{equation}

Now we prove that the result holds for $n = r$. The graph $B(m,r)$ can be considered as the union of the graphs $B(m,r-1)$ and $C_m$, whose intersection is $K_2$. Therefore by Equation~(\ref{eq3}), we have 
\begin{equation} \label{eq6}
\chi_{B(m,r)}(\lambda) = \frac{\chi_{B(m,r-1)}(\lambda)  \cdot \chi_{C_m}(\lambda)}{\lambda(\lambda-1)}.
\end{equation}
Using Equations~(\ref{eq4}) and ~(\ref{eq5}) in Equation~(\ref{eq6}), we get 
\begin{equation*}
\chi_{B(m,r)}(\lambda) = \lambda(\lambda-1) \gamma_{m}^{r}.
\end{equation*}
Hence the proof follows by induction.
\end{proof}

From Lemma~\ref{identical chro} and Theorem~\ref{B_0}, it follows that 
$$\chi_{B_0(m,n)}(\lambda)= \chi_{B(m,n)}(\lambda) =  \lambda(\lambda-1) \gamma_{m}^{n}.$$

It is clear that $B_{1}(m,1) \cong C_{m}^{-}$. Thus by Lemma~\ref{C-}, we have $\chi_{B_{1}(m,1)}(\lambda) = (\lambda - 1)^{m}$. We now compute the chromatic polynomials of the signed book graphs $B_{1}(m,n)$, for all $n$.
 
\begin{theorem} \label{B_1}
For $n \geq 1$, the chromatic polynomial of $B_{1}(m,n)$ is given by
\begin{equation*}
 \chi_{B_{1}(m,n)}(\lambda) = (\lambda-1)^m \gamma_{m}^{n-1}.
\end{equation*} 
\end{theorem}
\begin{proof}
Recall that $B_{1}(m,n)$ is the signed book graph with signature $\{uu_{1}^{1}\}$. For $n=1$, result holds true due to Lemma~\ref{C-}. For $n \geq 2$, consider $e_{1} = uu_{1}^{n}$. Using edge deletion-contraction formula on $e_{1}$, we get $$\chi_{B_{1}(m,n)}(\lambda) = \chi_{B_{1}'(m,n-1)}(\lambda) - \chi_{B_{1}''(m,n)}(\lambda),$$ where $B_{1}'(m,n-1)$ denotes the graph obtained from $B_{1}(m,n-1)$ by attaching a path $P_{m-1}$ at vertex $v$ and $B_{1}''(m,n)$ denotes the graph which is almost same as $B_{1}(m,n)$ but one of its pages is a cycle of length $m-1$ and the vertex obtained by contraction of $e_1$ is denoted by $u$ again. 

Using Lemma~\ref{Path_add}, the chromatic polynomial of the signed graph $B_{1}'(m,n-1)$ can be obtained in terms of $\chi_{B_{1}(m,n-1)}(\lambda)$. For $B_{1}''(m,n)$, we apply the edge deletion-contraction formula again on the edge $e_{2} = uu_{2}^{n}$. Repeated use of edge deletion-contraction formula and Lemma~\ref{Path_add} give the chromatic polynomial of $B_{1}(m,n)$ as
\begin{align*}
\chi_{B_{1}(m,n)}(\lambda)  = &(\lambda - 1)^{m-2} \chi_{B_{1}(m,n-1)}(\lambda) - (\lambda - 1)^{m-3} \chi_{B_{1}(m,n-1)}(\lambda)+ \cdots \cdots \\
 & + (-1)^{m-2} (\lambda - 1)^{(m-2)-(m-2)} \chi_{B_{1}(m,n-1)}(\lambda) \\
 = & \left[ \sum_{i=0}^{m-2} (-1)^{i} (\lambda - 1)^{(m-2)-i} \right] \chi_{B_{1}(m,n-1)}(\lambda)\\
  = & \gamma_{m}\chi_{B_{1}(m,n-1)}(\lambda) .
\end{align*}
Since $\chi_{B_{1}(m,1)}(\lambda) = (\lambda - 1)^{m}$, after $n-2$ recursions, we find that $$\chi_{B_{1}(m,n)}(\lambda ) = (\lambda-1)^m \gamma_{m}^{n-1}.$$

This completes the proof.
\end{proof}
 
\begin{theorem} \label{B_{1}*}
For $n \geq 2$, the chromatic polynomial of $B^{uv}(m,n)$ is given by
\begin{align*}
\chi_{B^{uv}(m,n)}(\lambda)  & =(\lambda - 1) \gamma_{m-1} \chi_{B^{uv}(m,n-1)}(\lambda) + (-1)^{m-2} (\lambda - 1)^{2} \gamma_{m}^{n-1}\\
 & = (\lambda - 1)^{m+n-1} \gamma_{m-1}^{n-1} + (-1)^{m-2} \gamma_m (\lambda - 1)^{2} \frac{(\lambda - 1)^{n-1} \gamma_{m-1}^{n-1} - \gamma_{m}^{n-1} }{(\lambda-1) \gamma_{m-1} - \gamma_{m}}.
\end{align*}
\end{theorem}
\begin{proof}
Recall that $B^{uv}(m,n)$ is the signed book graph with signature $\{uv\}$. Consider $e_{1} = uu_{1}^{n}$. Using edge deletion-contraction formula on $e_{1}$, we have $$\chi_{B^{uv}(m,n)}(\lambda) = \chi_{\widetilde{B^{uv}}(m,n-1)}(\lambda) - \chi_{\widetilde{\widetilde{B^{uv}}}(m,n)}(\lambda),$$
 where $\widetilde{B^{uv}}(m,n-1)$ denotes the graph obtained from $B^{uv}(m,n-1)$ by attaching a path $P_{m-1}$ at the vertex $v$ and $\widetilde{\widetilde{B^{uv}}}(m,n)$ denotes the graph which is almost same as $B^{uv}(m,n)$ but one of its pages is a cycle of length $m-1$.
  
Using Lemma~\ref{Path_add}, the chromatic polynomial of the signed graph $\widetilde{B^{uv}}(m,n-1)$ can be expressed in terms of $\chi_{B^{uv}(m,n-1)}(\lambda)$. For the graph $\widetilde{\widetilde{B^{uv}}}(m,n)$, we apply the edge deletion-contraction formula again on the edge $e_{2} = uu_{2}^{n}$. 

Applying edge deletion-contraction formula repeatedly and using Lemma~\ref{Path_add}, we obtain the chromatic polynomial of $B^{uv}(m,n)$ as
\begin{align*}
\chi_{B^{uv}(m,n)}(\lambda)  = &(\lambda - 1)^{m-2} \chi_{B^{uv}(m,n-1)}(\lambda) - (\lambda - 1)^{m-3} \chi_{B^{uv}(m,n-1)}(\lambda)+ \cdots \cdots  \\
& +(-1)^{m-3} (\lambda - 1)^{(m-2)-(m-3)} \chi_{B^{uv}(m,n-1)}(\lambda) + (-1)^{m-2} (\lambda - 1)^{(m-2)-(m-2)} \chi_{B_{m}^{n-1}}(\lambda)\\
   =& \Big[ \sum_{i=0}^{m-3} (-1)^{i} (\lambda - 1)^{(m-2)-i} \Big] \chi_{B^{uv}(m,n-1)}(\lambda) + (-1)^{m-2} \chi_{B_{m}^{n-1}}(\lambda).
\end{align*}
Note that, in the last step of edge deletion-contraction formula, the resulting graph is nothing but the signed graph $B_{m}^{n-1}$. We know by Proposition~\ref{prop_1} that $\chi_{B_{m}^{n-1}}(\lambda)= (\lambda - 1)^{2}\gamma_{m}^{n-1}$. Thus we have 
$$\chi_{B^{uv}(m,n)}(\lambda)  =(\lambda - 1) \gamma_{m-1} \chi_{B^{uv}(m,n-1)}(\lambda) + (-1)^{m-2} (\lambda - 1)^{2} \gamma_{m}^{n-1}.$$

Now, after $n-2$ recursions, we have 
\begin{align*}
\chi_{B^{uv}(m,n)}(\lambda) &  = (\lambda - 1)^{m+n-1} \gamma_{m-1}^{n-1}  + (-1)^{m-2} \frac{\gamma_{m}^{n+1}}{\gamma_{m-1}^{2}} \sum_{i=2}^{n} \left( \frac{ (\lambda-1) \gamma_{m-1}}{\gamma_m} \right)^{i} \\
  & = (\lambda - 1)^{m+n-1} \gamma_{m-1}^{n-1} + (-1)^{m-2} \gamma_m (\lambda - 1)^{2} \frac{(\lambda - 1)^{n-1} \gamma_{m-1}^{n-1} - \gamma_{m}^{n-1} }{(\lambda-1) \gamma_{m-1} - \gamma_{m}}.
\end{align*}
This completes the proof.
\end{proof}

We now give a formula for the chromatic polynomial of $B_{l}(m,n)$ for $2 \leq l \leq n-1 $.

\begin{theorem} \label{B_l}
Let $n \geq 3$ and $2 \leq l \leq n-1$. The chromatic polynomial of $B_{l}(m,n)$ is given by  
\begin{equation*}
 \chi_{B_{l}(m,n)}(\lambda) = \gamma_{m}^{n-l} \chi_{B^{uv}(m,l)}(\lambda).
\end{equation*} 
\end{theorem}
\begin{proof}
Repeated use of edge deletion-contraction formula and Lemma~\ref{Path_add} give that
\begin{equation}
\chi_{B_{l}(m,n)}(\lambda) = \frac{(\lambda - 1)^{m-1} + (-1)^{m}}{\lambda} \chi_{B_{l}(m,n-1)}(\lambda)= \gamma_m \chi_{B_{l}(m,n-1)}(\lambda). \label{final}
\end{equation} 
Applying the recursion of Equation~(\ref{final}) $(n-l-1)$ times, we have
\[\chi_{B_{l}(m,n)}(\lambda) = \gamma_{m}^{n-l} \chi_{B_{l}(m,l)}(\lambda).\]
Note that $B_{l}(m,l)$ is switching equivalent to $B^{uv}(m,l)$. Therefore we have
\[ \chi_{B_{l}(m,n)}(\lambda) = \gamma_{m}^{n-l} \chi_{B^{uv}(m,l)}(\lambda).\]
 This completes the proof.
\end{proof}

\section{Zero-free chromatic polynomials of signed book graphs}\label{balanced-poly}

In \cite{Zaslavsky2}, the author explained that the chromatic polynomial and the zero-free chromatic polynomial of $\Sigma$ are different unless $\Sigma$ is balanced. 

We reformulate Lemma~\ref{Path_add} in terms of the zero-free chromatic polynomial.

\begin{lemm} \label{Path_add1}
Let $\Sigma$ be a signed graph. Then 
$$\chi_{\Sigma^{t+1}}^{b}(\lambda) = (\lambda-1)^{t} \chi_{\Sigma}^{b}(\lambda).$$
\end{lemm}

\begin{theorem} \label{C_n}
For each $n \geq 2$, the zero-free chromatic polynomial of $C_{n}^{-}$ is given by 
$$\chi_{C_{n}^{-}}^{b}(\lambda) = (\lambda-1)^{n} - (-1)^{n} = \lambda \gamma_{n+1}.$$ 
\end{theorem}
\begin{proof}
The proof follows by induction on $n$.
\end{proof}

Now we present all the results of Section~\ref{chrom-poly} in terms of the zero-free chromatic polynomials. We omit their proofs as all these proofs are similar to the corresponding proofs presented for chromatic polynomials.

\begin{lemm} \label{B_m1}
For $m \geq 2$, the zero-free chromatic polynomial of $B_{m}^{1}$ is given by
 $$\chi_{B_{m}^{1}}^{b}(\lambda) = \lambda(\lambda-2)\gamma_m.$$
\end{lemm}

\begin{prop}\label{prop_1-1}
For $n \geq 2$, the zero-free chromatic polynomial of $B_{m}^{n}$ is given by 
\[ \chi_{B_{m}^{n}}^{b}(\lambda) =  \lambda(\lambda-2)\gamma_m^{n}.\]
\end{prop}

From Lemma~\ref{identical chro} and Theorem~\ref{B_0}, it follows that 
$$\chi_{B_0(m,n)}^{b}(\lambda) =  \lambda(\lambda-1) \gamma_{m}^{n}.$$

\begin{theorem} \label{B_1-1}
For $n \geq 1$, the zero-free chromatic polynomial of $B_{1}(m,n)$ is given by
\[ \chi_{B_{1}(m,n)}^{b}(\lambda) = \lambda \gamma_{m}^{n-1} \gamma_{m+1}.\]
\end{theorem}

\begin{theorem} \label{B_{1}*-1}
For $n \geq 2$, the zero-free chromatic polynomial of $B^{uv}(m,n)$ is given by
\begin{align*}
\chi_{B^{uv}(m,n)}^{b}(\lambda)  & =(\lambda - 1) \gamma_{m-1} \chi_{B^{uv}(m,n-1)}^{b}(\lambda) + (-1)^{m-2} \lambda(\lambda - 2) \gamma_{m}^{n-1}\\
 & = \lambda(\lambda - 1)^{n-1} \gamma_{m-1}^{n-1} \gamma_{m+1} + (-1)^{m-2} \lambda (\lambda - 2) \gamma_m \frac{(\lambda - 1)^{n-1} \gamma_{m-1}^{n-1} - \gamma_{m}^{n-1} }{(\lambda-1) \gamma_{m-1} - \gamma_{m}}.
\end{align*}
\end{theorem}
\begin{proof}
Proof of this theorem is similar to that of Theorem~\ref{B_{1}*}
\end{proof}

\begin{theorem} \label{B_l-1}
For $n \geq 2$ and $2 \leq l \leq n-1$, the zero-free chromatic polynomial of $B_{l}(m,n)$ is given by  
\[ \chi_{B_{l}(m,n)}^{b}(\lambda) = \gamma_m^{n-l} \chi_{B^{uv}(m,l)}^{b}(\lambda). \]
\end{theorem}

\section{Conclusion and future directions}\label{conclusion}

In this paper, we have determined recursive formulas for the chromatic polynomials and the zero-free chromatic polynomials of switching non-isomorphic signed book graphs. 

In ordinary graph theory, it is well known that the chromatic polynomial of union $G  \cup H$ of two graphs $G$ and $H$, where $G \cap H$ is a complete graph, is given by $$\chi_{G  \cup H}(k)  = \frac{\chi_{G}(k) \cdot \chi_{H}(k)}{\chi_{G  \cap H}(k)}.$$  

In the context of signed graph theory, we pose the following problem.

\begin{prob} \label{Problem-1}
\rm{Let $\Sigma_1$ and $\Sigma_2$ be two signed graphs such that $\Sigma_1 \cap \Sigma_2$ is a signed complete graph and that $\chi_{\Sigma_1}(\lambda),\chi_{\Sigma_2}(\lambda)~ \text{and}~ \chi_{\Sigma_1 \cap \Sigma_2}(\lambda)$ be known. What is the formula for $\chi_{\Sigma_1 \cup \Sigma_2}(\lambda)$?}  
\end{prob}

We pose the same problem for the zero-free chromatic polynomial of the union of two signed graphs.

\begin{prob} \label{Problem-2}
\rm{Let $\Sigma_1$ and $\Sigma_2$ be two signed graphs such that $\Sigma_1 \cap \Sigma_2$ is a signed complete graph and that $\chi^{b}_{\Sigma_1}(\lambda),\chi^{b}_{\Sigma_2}(\lambda)~ \text{and}~ \chi^{b}_{\Sigma_1 \cap \Sigma_2}(\lambda)$ be known. What is the formula for $\chi^{b}_{\Sigma_1 \cup \Sigma_2}(\lambda)$?}  
\end{prob}

In Problem~\ref{Problem-1} and Problem~\ref{Problem-2}, it is implicitly assumed that the signatures of $\Sigma_1$ and $\Sigma_2$ agree on their intersection part so that $\Sigma_1 \cup \Sigma_2$ is well defined.

Note that a signed book graph $(B(m,2),\sigma)$ is the union of two $m\text{-cycles}$ having exactly one edge in common, namely $uv$. The edge $uv$, being the intersection of two signed cycles, is indeed a signed complete graph on two vertices. Thus if we find answers of Problem~\ref{Problem-1} and Problem~\ref{Problem-2}, then the calculation of $\chi_{(B(m,2),\sigma)}(\lambda)$ and $\chi^{b}_{(B(m,2),\sigma)}(\lambda)$ are straightforward.  

Similarly, if the answers of Problem~\ref{Problem-1} and Problem~\ref{Problem-2} are known then for $n \geq 3$, both kind of chromatic polynomials of $(B(m,n),\sigma)$ can be obtained easily,  as $(B(m,n),\sigma)$ is the union of a signed $B(m,n-1)$ and a signed $m\text{-cycle}$, where their intersection is a signed complete graph on two vertices.

\noindent\textbf{Acknowledgements}

We sincerely thank the anonymous reviewers for carefully reading our manuscript and providing insightful comments and suggestions that helped us in substantial changes and improvements of the presentation of the article.

\end{document}